\documentclass{amsart}
\usepackage{amssymb}
\usepackage{hyperref}
\usepackage[initials]{amsrefs}
\usepackage{tikz-cd}
\usepackage[english]{babel}
\usepackage{tikz}
\usetikzlibrary{graphs,decorations.pathmorphing,decorations.markings}
\usepackage{graphicx}
\usepackage{amssymb, mathrsfs, enumerate}
\usepackage{mathtools}
\usepackage{xcolor}
\usepackage{comment}

\newcommand{\bbN}{{\mathbb N}}
\newcommand{\bbQ}{{\mathbb Q}}
\newcommand{\bbR}{{\mathbb R}}

\newcommand{\bbZ}{{\mathbb Z}}
\newcommand{\bbC}{{\mathbb C}}

\newcommand{\calH}{\mathcal{H}}

\newcommand{\calO}{\mathcal{O}}

\newcommand{\schlichting}{/\!/}

\newcommand{\covol}{\operatorname{covol}}

\newcommand{\defq}{\mathrel{\mathop:}=}
\newcommand{\id}{\operatorname{id}}

\newcommand{\res}{\operatorname{res}}

\newcommand{\SL}{\operatorname{SL}}

\newcommand{\Sym}{\operatorname{Sym}}

\newcommand{\PSL}{\operatorname{PSL}}

\newcommand{\GL}{\operatorname{GL}}

\newcommand{\FP}{\operatorname{FP}}
\newcommand{\KP}{\operatorname{KP}}
\newcommand{\Finite}{\operatorname{F}}

\newcommand{\topcat}[1]{\prescript{}{#1[G]}{\mathbf{top}}}
\newcommand{\discat}[1]{\prescript{}{#1[G]}{\mathbf{dis}}}

\newcommand{\dH}{{\operatorname{dH}}}

\newtheorem{theorem}{Theorem}[section]
\newtheorem{lemma}[theorem]{Lemma}

\newtheorem{corollary}[theorem]{Corollary}

\newtheorem{proposition}[theorem]{Proposition}
\newtheorem{prop}[theorem]{Proposition}

\theoremstyle{definition}
\newtheorem{definition}[theorem]{Definition}

\newtheorem{example}[theorem]{Example}

\newcommand{\comrs}[1]{\textcolor{red}{Roman:~#1}}

\begin{document}

\title{On homological properties of the Schlichting completion}
\date{February 2024}

\author{Laura Bonn}
\address{Karlsruhe Institute of Technology, Germany}
\email{laura.bonn@kit.edu}

\author{Roman Sauer}
\address{Karlsruhe Institute of Technology, Germany}
\email{roman.sauer@kit.edu}

\thanks{This research is funded by the Deutsche Forschungsgemeinschaft (DFG, German Research Foundation) – Project number 281869850. We thank Nicolas Monod for pointing out an error in the formulation of Theorem~\ref{thm: cohomology isomorphism}.}
\begin{abstract}
We show how finiteness properties of a group and a subgroup transfer to finiteness properties of the Schlichting completion relative to this subgroup. Further, we provide a criterion when the dense embedding of a discrete group into the Schlichting completion relative to one of its subgroups induces an isomorphism in (continuous) cohomology. As an application, we show that the continuous cohomology of the Neretin group vanishes in all positive degrees. 
\end{abstract}

\maketitle

\section{Introduction}

Finiteness conditions of discrete groups are higher-dimensional generalizations of the notions of being finitely generated and finitely presented. If a group satisfies suitable finiteness conditions, one can expect the group homology to enjoy nice properties, like, for example, being finitely generated. The appeal of studying finiteness conditions stems from the interaction between topology (specifically, the topology of classifying spaces) and algebra (notably, homological algebra involving chain complexes over the group ring). A similar theory for total disconnected locally compact Hausdorff groups, which we refer to as \emph{tdlc groups}, was still in its infancy a few years ago. What is different from the discrete case? On the topological side, tdlc groups often admit nice actions on CW-complexes or simplicial complexes but these actions are never free. On the algebraic side, smooth modules constitute a suitable abelian category; however, it possesses enough projectives only over the rationals.  

If one is content with studying finiteness conditions of a tdlc group~$G$ modulo the family of its compact-open subgroups, an elegant framework encompassing both topological and algebraic aspects becomes available. The class of $G$-equivariant CW-complexes with compact-open stabilizers enjoys a well-developed equivariant homotopy theory, similar to the discrete case~\cites{tomdieck, lueck}. Its algebraic counterpart, the category of chain complexes over the orbit category of $G$ with respect to the family of compact-open subgroups, is an abelian category, even when considered with integral coefficients~\cites{lueck}. Generally, however, it is unsatisfactory to be restricted to working modulo the family of compact-open subgroups. Thompson's group~$V$, for example, satisfies the finiteness condition $F_\infty$ in the usual sense -- an important early result of Geoghegan and Brown~\cite{brown+geoghegan} -- but not in the category of chain complexes of the orbit category with respect to the family of finite subgroups. There is a similar situation for Neretin's group, which is a totally disconnected analog of Thompson's group~$V$~\cite{sauer+thumann}. 

The recent work of Castellano and Corob Cook~\cite{ilaria+ged} drops all these limitations and establishes a convenient and elegant algebraic theory of finiteness conditions for tdlc groups, which works also with integral coefficients. Many of the fundamental properties of the discrete theory, as presented in Brown's foundational book~\cite{brown}, now find analogs in the study of tdlc groups. 

An important construction of tdlc groups from discrete groups is the Schlichting completion of a discrete group relative to a commensurated subgroup (see Section~\ref{sec: schlichting}). The contribution of this paper is to prove finiteness properties of the Schlichting completions and to relate finiteness properties and cohomology of the Schlichting completion to the ones of its defining discrete group. 

If the commensurated subgroup is normal, then the Schlichting completion is just the quotient, in particular, it is discrete. One should read Theorems~\ref{thm: finiteness for discrete group},~\ref{thm: finiteness for tdlc group} and~\ref{thm: euler characteristic} below with this in mind; the results for quotient groups are well known. 

The definitions of properties  $\FP^R_n$ and $\Finite_n$ are recalled in Section~\ref{sec: finiteness}.  

\begin{theorem}\label{thm: finiteness for discrete group}
    Let $G=\Gamma\schlichting\Lambda$ be the Schlichting completion of~$\Gamma$ relative to the commensurated subgroup $\Lambda<\Gamma$. Let $R$ be a commutative ring. Then the following holds. 
    \begin{enumerate}
        \item If $\Lambda$ and $G$ have type~$\FP^R_n$, then $\Gamma$ has type $\FP^R_n$.
        \item If $\Lambda$ and $G$ have type~$\Finite_n$, then $\Gamma$ has type $\Finite_n$.  
    \end{enumerate}
    \end{theorem}

\begin{theorem}\label{thm: finiteness for tdlc group}
    Let $G=\Gamma\schlichting\Lambda$ be the Schlichting completion of~$\Gamma$ relative to the commensurated subgroup $\Lambda<\Gamma$. Let $R$ be a commutative ring. Then the following holds. 
    \begin{enumerate}
        \item If $\Lambda$ has type $\FP^R_{n-1}$ and $\Gamma$ has type~$\FP^R_n$, then $G$ has type $\FP^R_n$.
        \item If $\Lambda$ has type $\Finite_{n-1}$ and $\Gamma$ has type~$\Finite_n$, then $G$ has type $\Finite_n$.  
    \end{enumerate}
    \end{theorem}

Theorems~\ref{thm: finiteness for discrete group} and~\ref{thm: finiteness for tdlc group} are proved in Section~\ref{sec: finiteness}. 

It is an interesting question when the restriction map from the continuous cohomology of a locally compact group to the cohomology of a dense subgroup is an isomorphism. For the inclusion $\SL_n(\bbQ)\hookrightarrow \SL_n(\bbR)$ this was proved by Borel-Yang~\cite{borel+yang} in order to solve the rank conjecture in algebraic K-theory. In the next result, which is proved in Section~\ref{sec: continuous cohom}, we consider the easier situation of the inclusion of a discrete group into its Schlichting completion. 

\begin{theorem}\label{thm: cohomology isomorphism}
    Let $G=\Gamma\schlichting\Lambda$ be the Schlichting completion of~$\Gamma$ relative to a locally finite commensurated subgroup $\Lambda<\Gamma$. Then the restriction map $H^\ast_c(G,\bbR)\to H^\ast(\Gamma, \bbR)$ is an isomorphism in all degrees.   
\end{theorem}

Neretin's group $N_d$, which is the group of almost automorphisms of a non-rooted $(d+1)$-regular tree, is the Schlichting completion of the Higman-Thompson's group~$V_{d,2}$ relative to a locally finite commensurated subgroup~\cite{boudec}*{Example~6.7}.     
Brown~\cite{brown-acyclic} showed  the rational acyclicity\footnote{Szymik and Wahl proved the much stronger integral acyclicity~\cite{szymik+wahl}.} of~$V_{d,2}$. 

We obtain the following consequence. 

\begin{corollary}
    Let $d\ge 2$. The continuous cohomology $H^i_c(N_d,\bbR)$ of Neretin's group~$N_d$  vanishes for every $i>0$. 
\end{corollary}

In the next result, $\chi^{(2)}(G, \mu)$ denotes the Euler characteristic of an unimodular tdlc group. This invariant is discussed in Section~\ref{sec: Euler}. If $G$ is a discrete group with a finite model of its classifying space and $\mu$ is the counting measure, then $\chi^{(2)}(G, \mu)$ is the usual Euler characteristic. If $G$ is discrete and has torsion, it is the $\ell^2$-Euler characteristic whenever it is defined.

\begin{theorem}\label{thm: euler characteristic}
    Let $G=\Gamma\schlichting\Lambda$ be the Schlichting completion of~$\Gamma$ relative to the commensurated subgroup $\Lambda<\Gamma$. Suppose that $G$ is unimodular and that $\Lambda$ and $\Gamma$ have type $\FP^\bbQ$. Then $G$ has type $\FP^\bbQ$ and we have 
    \[ \chi^{(2)}(\Lambda)\cdot\chi^{(2)}(G,\mu)=\chi^{(2)}(\Gamma)\]
    for the Haar measure $\mu$ with $\mu(U)=1$ where $U<G$ is the closure of $\Lambda$. 
\end{theorem}

\section{The Schlichting completion}\label{sec: schlichting}

The Schlichting completion of a discrete group $\Gamma$ relative to the commensurated subgroup $\Lambda$ is a tdlc group which we denote by $G=\Gamma \schlichting\Lambda$. 
This construction was introduced in~\cite{tzanev}, following an earlier idea of Schlichting~\cite{schlichting}. 

A nice background reference is the work of Shalom and Willis~\cite{shalom+willis} who call the Schlichting completion the  \emph{relative profinite completion of $\Gamma$ with respect to $\Lambda$}.

Let $\Gamma$ be a discrete group and $\Lambda < \Gamma$ be a commensurated subgroup. 
Then $\Gamma$ acts by left multiplication on $\Gamma /\Lambda$ and thus defines a homomorphism \[\alpha \colon \Gamma \to \Sym(\Gamma/\Lambda).\] 
We equip $\Sym(\Gamma/\Lambda)$ with the topology of pointwise convergence.  
The closure \[G=\Gamma \schlichting \Lambda= \overline{\alpha(\Gamma)}\] is the \emph{Schlichting completion of $\Gamma$ relative to the commensurated subgroup $\Lambda$}.
Strictly speaking, the Schlichting completion is not a completion of~$\Gamma$ since $\alpha$ might not be injective. 

In the following, we collect some properties of this construction. 

\begin{prop}[\cite{shalom+willis}*{Section~3}]
    Let $G=\Gamma\schlichting\Lambda$ be the Schlichting completion of~$\Gamma$ relative to the commensurated subgroup $\Lambda<\Gamma$.
    \begin{enumerate}   
        \item If $\Lambda$ is normal in $\Gamma$ then $G=\Gamma/\Lambda$.
        \item $G$ is a tdlc group. 
        \item The map $\alpha \colon \Gamma \to G$ has a dense image. Its kernel is the largest subgroup that is normal in $\Gamma$ and contained in $\Lambda$.  
        \item The closure of the image $\overline{\alpha(\Lambda)}$ is a compact open subgroup of $G$. In particular, it is commensurated in $G$. 
    \end{enumerate}
\end{prop}
For the proof of Theorem~\ref{thm: compactly presented} we use the following easy fact about the Schlichting completion. 

\begin{lemma}[\cite{boudec}*{Lemmas~6.3 and~6.4}]\label{lem: Le Boudec}
The following holds for the Schlichting completion. 
\begin{enumerate}
\item 	    $\Gamma\schlichting\Lambda=\overline{\alpha(\Lambda)}\alpha(\Gamma)$;
\item     $\overline{\alpha(\Lambda)}\cap \alpha(\Gamma)=\alpha(\Lambda)$.
\end{enumerate}
\end{lemma}

An easy consequence is that $\alpha$ induces an isomorphism $\Gamma/\Lambda\xrightarrow{\cong} G/U$ where $G=\Gamma\schlichting\Lambda$ and $U=\overline{\alpha(\Lambda)}$. This fact will be used frequently. 


\begin{theorem}\label{thm: compactly presented}
    Let $G=\Gamma \schlichting\Lambda$ be the Schlichting completion of $\Gamma $ relative to the commensurated subgroup $\Lambda < \Gamma$. 
    \begin{enumerate}
        \item If $G$ is compactly generated and $\Lambda$ is finitely generated, then $\Gamma$ is finitely generated. 
        \item If $G$ is compactly presented and $\Lambda$ is finitely presented, then $\Gamma$ is finitely presented. 
        \item If $\Gamma$ is finitely generated, then $G$ is compactly generated. 
        \item If $\Gamma$ is finitely presented and $\Lambda$ is finitely generated, then $G$ is compactly presented~\cite{boudec}*{Theorem~6.1}. 
    \end{enumerate}
\end{theorem}

Before we prove this, we consider 
the following two propositions which show that Theorem~\ref{thm: compactly presented} implies Theorem~\ref{thm: finiteness for discrete group} part (2) and Theorem~\ref{thm: finiteness for tdlc group} part (2) for the cases $n=1,2$. For the notion of type $F_n$ see Definition~\ref{def: finiteness properties}. 

\begin{proposition}[\cite{geoghegan}*{Proposition 7.2.1}]\label{prop: compactly gen via CW}
Let $G$ be a discrete group. Then the following equivalences hold. 
\begin{enumerate}
    \item $G$ is of type $F_1$ if and only if $G$ is finitely generated. 
    \item $G$ is of type $F_2$ if and only if $G$ is finitely presented.
\end{enumerate}
\end{proposition}

The analogous proposition can be formulated for tdlc groups. 

\begin{proposition}[\cite{ilaria+ged}*{Proposition 3.4}]\label{prop: compactly gen via G-CW}
    Let $G$ be a tdlc group. Then the following equivalences hold. 
\begin{enumerate}
    \item $G$ is of type $F_1$ if and only if $G$ is compactly generated. 
    \item $G$ is of type $F_2$ if and only if $G$ is compactly presented.
\end{enumerate}
\end{proposition}

\begin{proof}[Proof of Theorem~\ref{thm: compactly presented}]
The proof of (3) immediate since the union of a generating set of $\Gamma$ and a compact-open subgroup of $G$ is a generating set of~$G$. As indicated in the statement, (4) is proved by Le Boudec. 

Next we prove (1) and (2). 
If $G$ is compactly generated or compactly presented, then $G$ is of type $F_1$ or $F_2$, respectively. 
Being of type $F_1$ or $F_2$ is witnessed by a contractible proper smooth $G$-CW complex $X$ with cocompact $1$-skeleton or $2$-skeleton. 
Since the stabilizers of $G$ are compact-open, they are commensurable with $\alpha(\Lambda)$. It follows from  Lemma~\ref{lem: Le Boudec} that the stabilizers of the restricted $\Gamma$-action on~$X$ are commensurable with $\Lambda$. 
In the first case the stabilizers of the $\Gamma$-action are finitely generated, in the second case they are finitely presented. 
Therefore, $\Gamma$ is finitely generated by a special case of the Schwarz-Milnor lemma or finitely presented by a theorem of Brown~\cite{brown-finiteness}*{Proposition~3.1}, respectively. 
\end{proof}

\section{Finiteness properties of the Schlichting completion}\label{sec: finiteness}
Finiteness properties of tdlc groups over $\bbQ$ were introduced and studied by Castellano-Weigel~\cite{castellano+weigel}. Castellano-Corob Cook developed a theory of finiteness properties of tdlc groups over an arbitrary commutative ground ring~\cite{ilaria+ged} which we briefly review first. 

A natural setting for the homological algebra of tdlc groups is the category $\discat{R}$ of \emph{discrete $R[G]$-modules}, that is, of $R$-modules equipped with a left action of $G$ such that the stabilizer of each element is open. A discrete $R[G]$-module of the form $R[\Omega]$ where $\Omega$ is a discrete set with a continuous $G$-action is called a \emph{discrete permutation $R[G]$-module}. This means that the continuous $G$-action on $\Omega$ has open stabilizers. If the stabilizers are also compact, then $R[\Omega]$ is \emph{proper}. 

The category $\discat{R}$ is an abelian category that has enough injectives. If $\bbQ\subset R$, then $\discat{R}$ also has enough projectives, and every proper discrete permutation $R[G]$-module is projective. For $R=\bbZ$ this is no longer true in general. For any $R$, the category $\discat{R}$ embeds into a quasi-abelian category $\topcat{R}$ that has enough projectives. Although proper discrete permutation $R[G]$-modules are not necessarily projective for arbitrary rings $R$ we still have the equivalence (2) in  Theorem~\ref{thm: relations between different finiteness conditions}. As a consequence, a reader of this paper does not really have to know what $\topcat{R}$ and $\KP_n^R$ are and can just work with the more intuitive notions $\discat{R}$ and $\FP_n^R$. However, the reason that the notion of $\FP_n^R$ works well, as in e.g.~Proposition~\ref{prop: finiteness of short exact sequences}, is that there is the quasi-abelian category $\topcat{R}$ in the background.

\begin{definition}\label{def: finiteness properties}
Let $R$ be a commutative ring and $n\in\bbN\cup\{\infty\}$. We say that a tdlc group~$G$ has 
\begin{enumerate}
\item type $\Finite_n$ if there is a contractible proper smooth $G$-CW-complex with cocompact $n$-skeleton;
\item type $\FP_n^R$ if the trivial $R[G]$-module $R$ has a resolution $P_\ast\to R$ by proper discrete permutation $R[G]$-modules $P_\ast$ such that $P_0,\dots,P_n$ are finitely generated; 
\item type $\KP_n^R$ if the trivial $R[G]$-module $R$ has a projective resolution $P_\ast$ in the category $\topcat{R}$ such that $P_0,\dots, P_n$ are compactly generated. 
\end{enumerate}
\end{definition}

For the definition of $G$-CW-complexes see~\cite{tomdieck}. A $G$-CW-complex is proper or smooth if all its stabilizers are compact or open subgroups, respectively. 

\begin{theorem}[Castellano-Corob Cook]\label{thm: relations between different finiteness conditions}
    Let $G$ be a tdlc group. Let $R$ be a commutative ring and $n\in\bbN\cup\{\infty\}$. 
    Then the following holds. 
    \begin{enumerate}
        \item If $G$ is compactly presented and $G$ has type $\FP_n^\bbZ$ then $G$ has type $\Finite_n$~\cite{ilaria+ged}*{Proposition~3.13}. 
        \item The group $G$ has type $\FP_n^R$ if and only if $G$ has type $\KP^R_n$~\cite{ilaria+ged}*{Theorem~3.10}. 
    \end{enumerate}
\end{theorem}

\begin{lemma}\label{lem: restriction}
Let $G=\Gamma\schlichting\Lambda$ be the Schlichting completion of~$\Gamma$ relative to the commensurated subgroup $\Lambda<\Gamma$. Let $M=R[\Omega]$ be a finitely generated proper discrete permutation $G$-module over $R$. If $\Lambda$ is of type $\FP^R_n$, then $\res^G_\Gamma(M)$ has a projective $R[\Gamma]$-resolution $P_\ast\twoheadrightarrow M$ such that $P_0, \dots, P_n$ are finitely generated. If $\Lambda$ is locally finite and $\bbQ\subset R$, then $\res^G_\Gamma(M)$ is a flat $R[\Gamma]$-module. 
\end{lemma}

\begin{proof}
If $\Lambda$ has type $\FP^R_n$ then so does any subgroup of $\Gamma$ that is commensurated with $\Lambda$ by~\cite{brown}*{(5.1) Proposition on p.~197}. Let $\Lambda'<\Gamma$ commensurated with $\Lambda$. 
Let $Q_\ast\to R$ be a projective $R[\Lambda']$-resolution of the trivial module such that $Q_0, \dots, Q_n$ are finitely generated. 
Then $R[\Gamma]\otimes_{R[\Lambda']} Q_\ast$ is a projective resolution of $R[\Gamma/\Lambda']$ that is finitely generated in degrees $0,\dots, n$. 
Hence the $R[\Gamma]$-module $R[\Gamma/\Lambda']$ has type $\FP^R_n$. A finitely generated discrete permutation $G$-module $R[\Omega]$ is a finite sum of modules of the type $R[G/U]$ where $U<G$ is a compact-open subgroup. 
By Lemma~\ref{lem: Le Boudec} we have $G/U\cong \Gamma/\alpha^{-1}(U)$, and $\alpha^{-1}(U)$ is commensurable with $\alpha^{-1}(\overline{\alpha(\Lambda)})=\Lambda$. 
Therefore $\res^G_\Gamma R[G/U]$ is of type $\FP^R_n$. 

If $\Lambda$ and thus $\Lambda'$ are locally finite and $\bbQ\subset R$, then $R$ is a flat $R[\Lambda']$-module~\cite{bieri}*{Proposition~4.12~on~p.~63}. Therefore $R[\Gamma]\otimes_{R[\Lambda']} R=R[\Gamma/\Lambda']$ is a flat $R[\Gamma]$-module.  
\end{proof}

\begin{lemma}[\cite{brown2}*{Lemma~1.5}]\label{lem: double complex lemma}
Let $C_\ast$ be a chain complex over a ring. Let $P_\ast^{(i)}$ be a projective resolution of $C_i$. Then there is a chain complex $Q_\ast$ with $Q_n=\bigoplus_{i+j=n} P_i^{(j)}$ and a weak equivalence $Q_\ast\to C_\ast$. 
\end{lemma}

\begin{proof}[Proof of Theorem~\ref{thm: finiteness for discrete group}]
    We only need to prove part (1), because part (2) follows directly from part (1), Theorem \ref{thm: relations between different finiteness conditions} and Theorem \ref{thm: compactly presented}.
    
    Let \[\dots\to P_{n+1}\to P_n\to\dots \to P_0\to R\to 0\] be a projective resolution of the trivial $G$-module by proper discrete permutation modules such that $P_0,\dots, P_n$ are finitely generated. 
      By Lemma~\ref{lem: restriction} each $R[\Gamma]$-module $\res^G_\Gamma(P_j)$, $j\le n$, has a projective resolution $Q_\ast^{(j)}$ such that $Q_i^{(j)}$ is finitely generated for $i\in\{0,\dots,n\}$. For $j>n$ let be $Q_\ast^{(j)}$ any projective resolution of $\res^G_\Gamma(P_j)$. By Lemma~\ref{lem: double complex lemma} there is a projective resolution $Q_\ast$ of the trivial $R[\Gamma]$-module $R$ such that 
      \[ Q_k=\bigoplus_{i+j=k}Q_i^{(j)},\]
      which concludes the proof. 
      \end{proof}

The following proposition follows from combining Proposition~3.9 and Theorem~3.10 in~\cite{ilaria+ged}.

\begin{prop}[Castellano-Corob Cook]\label{prop: finiteness of short exact sequences}
Let $G$ be a tdlc group and $R$ be a commutative ring. 
Let $0\to A'\to A\to A''\to 0$ a short exact sequence of discrete $R[G]$-modules. Then the following statements hold true. 
\begin{enumerate}[(a)]
\item If $A'$ has type $\FP^R_{n-1}$ and $A$ has type $\FP^R_n$, then $A''$ has type $\FP^R_n$. 
\item If $A$  has type $\FP^R_{n-1}$ and $A''$ has type $\FP^R_n$, then $A'$ has type $\FP^R_{n-1}$.
\item If $A'$ and $A''$ have type $\FP^R_n$, then so does $A$.  
\end{enumerate}
\end{prop}

\begin{prop}\label{prop: more general finiteness result for modules}
Let $G=\Gamma\schlichting\Lambda$ be the Schlichting completion of~$\Gamma$ relative to the commensurated subgroup $\Lambda<\Gamma$. Let $R$ be a commutative ring. Let $M$ be a discrete $R[G]$-module. 
If $\Lambda$ has type $\FP^R_m$ and $\res^G_\Gamma(M)$ has type $\FP^R_n$ then $M$ has type $\FP^R_{\min\{m+1,n\}}$. 
\end{prop}

\begin{proof}
If $\res_\Gamma^G(M)$ is not finitely generated, we are done. If $\res_\Gamma^G(M)$ is finitely generated, then $M$ is clearly finitely generated. In particular, there is a short exact sequence
\begin{equation}\label{eq: short exact presentation}
0\to K\to P\to M\to 0,
\end{equation}
where $P$ is a finitely generated proper discrete permutation module. 

We show the statement by induction over~$n$. 
The case $n=0$ just means finite generation, and there is nothing more to do. 
Suppose the statement holds true for every restriction of an $R[G]$-module of type $\FP^R_{n-1}$. Let $\res_\Gamma^G(M)$ be of type $\FP^R_n$ and choose a sequence as in~\eqref{eq: short exact presentation}. 
We apply Proposition~\ref{prop: finiteness of short exact sequences} to the short exact sequence~\eqref{eq: short exact presentation} for the tdlc group~$G$ and to the short exact sequence 
\begin{equation}\label{eq: second exact sequence} 0\to \res^G_\Gamma(K)\to \res^G_\Gamma(P)\to \res^G_\Gamma(M)\to 0
\end{equation}	
for the discrete group $\Gamma$. 
By Lemma~\ref{lem: restriction} the module $\res^G_\Gamma(P)$ has type $\FP^R_m$.  
By part~(b) of the above proposition the kernel $\res_\Gamma^G(K)$ has type $\FP^R_{\min\{m,n-1\}}$. By induction hypothesis, $K$ has type $\FP^R_{\min\{m,n-1\}}$. By part~(a) of the above proposition, applied to~\eqref{eq: short exact presentation}, we obtain that $M$ has type $\FP^R_{\min\{m,n-1\}+1}=\FP^R_{\min\{m+1,n\}}$. This concludes the proof. 
\end{proof}

\begin{proof}[Proof of Theorem~\ref{thm: finiteness for tdlc group}]
The first part of the theorem follows by applying Proposition~\ref{prop: more general finiteness result for modules} to the trivial $G$-module~$R$. If $\Lambda$ is compactly generated and $\Gamma$ is compactly presented, then $G$ is compactly presented by Theorem~\ref{thm: compactly presented}. By~\cite{ilaria+ged}*{Proposition 3.13} being compactly presented and having type $\FP^\bbZ_n$ is equivalent to having type $\Finite_n$. Therefore the second part of the theorem follows from the first one. 
\end{proof}

\begin{example}[The Abels-Brown group]
Let $R$ be a commutative ring. Let $\Gamma_n(R)$ denote the subgroup of $\GL_{n+1}(R)$ that consists of upper triangular matrices $(g_{i,j})$ such that $g_{1,1}=g_{n+1,n+1}=1$. For example, $\Gamma_2(R)$ consists of matrices of the form 
\[
\begin{pmatrix}
    1 & \ast & \ast\\
    0 & \ast & \ast\\
    0 & 0    & 1
\end{pmatrix}.
\]
This group was studied by Abels and Brown~\cite{abels+brown}. They showed that $\Gamma_n(\bbZ[1/p])$ is of type $\FP^\bbZ_{n-1}$ but not of type $\FP^\bbZ_n$. Moreover, for $n\ge 3$ it is finitely presented. 
The subgroup $\Lambda_n=\Gamma_n(\bbZ)$ has entries $\pm 1$ on the diagonal. Therefore, $\Lambda_n$ is finitely generated nilpotent, hence of type $\FP^\bbZ_\infty$. Let $G_n$ be the Schlichting completion of $\Gamma_n(\bbZ[1/p])$ relative to $\Lambda_n$. By Theorem~\ref{thm: finiteness for tdlc group}, $G_n$ is of type $\FP^\bbZ_{n-1}$. By Theorem~\ref{thm: finiteness for discrete group}, $G_n$ is not of type $\FP^\bbZ_n$. 
By Theorem~\ref{thm: compactly presented}, $G_n$ is compactly presented for $n\ge 3$. By~\cite{shalom+willis}*{Lemma~3.5 and~3.6} $G_n$ is isomorphic to $\Gamma_n(\bbQ_p)/K$ where $K$ is the largest normal subgroup of $\Gamma_n(\bbQ_p)$ that is contained in  $\Gamma_n(\bbZ_p)$. Since $K$ is compact we can conclude that $\Gamma_n(\bbQ_p)$ has the same finiteness properties as $G_n$ by~\cite{ilaria+ged}*{Lemma~3.15}. 
\end{example}

\section{Continuous cohomology vs. cohomology of the dense subgroup}\label{sec: continuous cohom}

The \emph{continuous cohomology} of a locally compact group is defined by the complex of continuous cochains in the standard resolution. In Proposition~\ref{prop: continuous vs discrete cohomology} below, we compare the continuous cohomology to the \emph{discrete cohomology} both with real coefficients. The discrete cohomology can be computed from a resolution by proper discrete permutation modules and was introduced as a derived functor in~\cite{castellano+weigel}*{Section~2.5}. 

We do not claim originality for Proposition~\ref{prop: continuous vs discrete cohomology}. It can be deduced from the results in Guichardet's book~\cite{guichardet} but we give a proof because we need the specific chain map $\phi$ used in the proof later.  

\begin{prop}\label{prop: continuous vs discrete cohomology}
The continuous cohomology $H_c^\ast(G,\bbR)$ of a tdlc group~$G$ is 
isomorphic to the real discrete cohomology $\dH^\ast(G,\bbR)$. 
\end{prop}

\begin{proof}
Let $U<G$ be a compact-open subgroup. 
Let $\mu$ be the left-invariant Haar measure on $G$ with $\mu(U)=1$. 
Then $\bbR[(G/U)^{\ast+1}]$ with the usual differentials of the bar resolution is a resolution of the trivial $G$-module $\bbR$ by proper discrete permutation modules. It suffices to show that the projection $G\to G/U$ induces a homotopy equivalence
\begin{equation}\label{eq: chain map discrete vs continuous} \phi\colon \hom_{\bbR[G]}\bigl(\bbR[(G/U)^{\ast+1}],\bbR\bigr)^G\to C\bigl(G^{\ast+1},\bbR\bigr)^G.
\end{equation}
A homotopy inverse $\rho$ is defined as follows. For a cochain $f\colon G^{n+1}\to \bbR$ let $\rho(f)\colon (G/U)^{n+1}\to\bbR$ be the map 
\[ \rho(f)(g_0U,\dots,g_nU)=\int_{U^{n+1}} f\bigl( g_0u_0,\dots, g_nu_n \bigr)d\mu(u_0)\dots d\mu(u_n). \]
Since $U^{n+1}$ is compact and $f$ is continuous the integral exists. The definition is independent of the choice of representatives $g_0,\dots, g_n$ of the $U$-coset classes by the left-invariance of~$\mu$. Clearly, $\rho$ is a cochain map and $\rho\circ\phi=\id$. 

The chain homotopy $\phi\circ\rho\simeq \id$ is defined as follows. Let  
\[ S^n_i(f)(g_0,\dots, g_{n-1})\defq\int_{U^i} f\bigl(g_0u_0,\dots,g_{i-1}u_{i-1},g_i,\dots,g_{n-1}\bigr)d\mu(u_0)\dots d\mu(u_i)\]
for every $n\ge 1$ and every $0\le i\le n-1$.  
Similarly as above, this formula defines a homomorphism $S^n_i\colon C(G^{n+1}, \bbR)^G\to C(G^n, \bbR)^G$. Then $H^n=\sum_{i=0}^{n-1} (-1)^iS^n_i$ is the chain homotopy $\phi\circ\rho\simeq \id$. 
\end{proof}

Now we are able to quickly conclude the proof of Theorem~\ref{thm: cohomology isomorphism}. 

\begin{proof}[Proof of Theorem~\ref{thm: cohomology isomorphism}]
Let $G=\Gamma\schlichting\Lambda$ and $U$ be the closure of $\Lambda$ in~$G$. 
Let $P_\ast=\bbR[(G/U)^{\ast+1}]$ be the resolution of the trivial $G$-module $\bbR$ appearing in the proof of Proposition~\ref{prop: continuous vs discrete cohomology}. Each $P_n$ is a proper discrete permutation module. By Lemma~\ref{lem: restriction}, the restricted resolution  
$\res^G_\Gamma P_\ast=\bbR[(\Gamma/\Lambda)^{\ast+1}]$ is a flat $\bbR[\Gamma]$-resolution of~$\bbR$. 

 The map~$\psi$ in the following commutative square is induced by the projection $\Gamma\to \Gamma/\Lambda$. The map $\phi$ is the one in~\eqref{eq: chain map discrete vs continuous}. 

\[
\begin{tikzcd}
C\bigl(G^{\ast+1}, \bbR\bigr)^G\ar[r, "\res"] & \hom\bigl(\bbR[\Gamma^{\ast+1}],\bbR)^\Gamma  \\
\hom\bigl(P_\ast, \bbR\bigr)^G\ar[r,"\cong"]\ar[u, "\phi"] &  \hom\bigl(\res^G_\Gamma P_\ast, \bbR\bigr)^\Gamma \ar[u,"\psi"]
\end{tikzcd}
\]
The statement of Theorem~\ref{thm: cohomology isomorphism} is that the upper horizontal restriction is a weak isomorphism. The map $\phi$  is a weak isomorphism by the proof of Proposition~\ref{prop: compactly gen via CW}. The forgetful lower horizontal map is obviously an isomorphism. 
So it suffices to show that $\psi$ is a weak isomorphism. 

By~\cite{weibel}*{Lemma~3.2.8 on p.~71} the projection from the projective resolution $\bbR[\Gamma^{\ast+1}]$
to the flat resolution $\res^G_\Gamma P_\ast$ induces a weak isomorphism \[ \bbR[\Gamma^{\ast+1}]\otimes_{\bbR[\Gamma]}\bbR\xrightarrow{\sim} \res^G_\Gamma P_\ast\otimes_{\bbR[\Gamma]}\bbR.\] Its dual map 
\[ 
\hom_\bbR\bigl(\res^G_\Gamma P_\ast\otimes_{\bbR[\Gamma]}\bbR,\bbR\bigr)\xrightarrow{\sim} \hom_\bbR\bigl(\bbR[\Gamma^{\ast+1}]\otimes_{\bbR[\Gamma]}\bbR,\bbR   \bigr)
\]
is isomorphic to~$\psi$. 
The dual map is a weak isomorphism by the universal coefficient theorem over $\bbR$~\cite{weibel}*{Theorem~3.6.5 on p.~89}. 
\end{proof}

\section{The Euler characteristic of the Schlichting completion}\label{sec: Euler}

The Euler characteristic $\chi^{(2)}(G,\mu)\in \bbR$ of a unimodular tdlc group~$G$ with Haar measure~$\mu$ that admits a contractible smooth proper $G$-CW-complex or a finite resolution by proper discrete permutation modules was introduced in~\cite{petersen+sauer+thom}. A more general  approach can be found in~\cite{euler}.

If 
\begin{equation}\label{eq: resolution permutation} 0\to \bigoplus_{i\in I_n} \bbC[G/U_i^{(n)}]\to \dots\to \bigoplus_{i\in I_0}\bbC[G/U_i^{(0)}]\to \bbC\to 0\end{equation}
is a resolution by proper discrete permutation modules, then 
\begin{equation}\label{eq: def of Euler} \chi^{(2)}(G,\mu)=\sum_{p=0}^n (-1)^p \sum_{i\in I_p} \mu\bigl(U_i^{(p)}\bigr)^{-1}.
\end{equation}
If $G$ is discrete, then one usually takes the counting measure as Haar measure and omits the Haar measure in the notation. In this case, $\chi^{(2)}(G,\mu)$ coincides with the $\ell^2$-Euler characteristic of the group~\cite{lueck}*{Section~7.2}. Moreover, if $G$ is discrete and of type $F$, then $\chi^{(2)}(G,\mu)$ coincides with the classical Euler characteristic of the group. 

By~\cite{petersen+sauer+thom}*{Theorem~4.9} the Euler characteristic of a unimodular tdlc group is the alternating sum of its $\ell^2$-Betti numbers, which were introduced by Petersen~\cite{petersen}.  
\begin{equation}\label{eq: alternating l2 betti} \chi^{(2)}(G,\mu)=\sum_{p\ge 0}(-1)^p \beta^{(2)}_p(G,\mu)
\end{equation}
The terms in~\eqref{eq: def of Euler} can be interpreted in terms of the von Neumann dimensions $\dim_G$ of modules of the type $L(G,\mu)\otimes_{\calH(G)}\bbC[G/U]=L(G,\mu)p_U$, that is, 
\[ \dim_G L(G,\mu)\otimes_{\calH(G)}\bbC[G/U]=\mu(U)^{-1},\]
where $L(G,\mu)$ is the von Neumann algebra of $G$ relative to $\mu$, and $\calH(G)$ is the Hecke algebra of complex-valued locally constant functions, and $L(G,\mu)p_U$ is the projection onto the $U$-invariant vectors in $L^2(G,\mu)$. The proof of the second formula of $\chi^{(2)}(G, \mu)$ is then just a matter of additivity of the von Neumann dimension. We refer to~\cite{petersen+sauer+thom} for more details. 

An important consequence of~\eqref{eq: alternating l2 betti} and the corresponding property for $\ell^2$-Betti numbers~\cite{kyed+petersen+vaes} is the equality 
\begin{equation}\label{eq: multiplicativity for lattices}
\chi^{(2)}(\Gamma)=\covol(\Gamma,\mu)\cdot \chi^{(2)}(G,\mu)
\end{equation}
for every lattice $\Gamma$ in a locally compact group $G$ with Haar measure~$\mu$. In particular, if $\Lambda<\Gamma$ is a subgroup of finite index in a discrete group $\Gamma$ of type $\FP^\bbQ$, then 
\begin{equation}\label{eq: multiplicativity for finite index}\chi^{(2)}(\Lambda)=[\Gamma:\Lambda]\cdot \chi^{(2)}(\Gamma).
\end{equation}

\begin{proof}[Proof of Theorem~\ref{thm: euler characteristic}]
    As before, we denote the canonical map of $\Gamma$ into the Schlichting completion $G=\Gamma\schlichting\Lambda$ by $\alpha$. Further, $U$ is the closure of $\alpha(\Lambda)$. 
    Consider a projective resolution of the trivial $G$-module by proper discrete permutation modules as in~\eqref{eq: resolution permutation}. 
For every $j\in\{0,\dots,n\}$ and every $i\in I_j$ we choose a finite projective $\bbC[\alpha^{-1}(U_i^{(j)})]$-resolution $\tilde Q(i,j)_\ast$ of the trivial $\bbC[\alpha^{-1}(U_i^{(j)})]$-module~$\bbC$. Since $\alpha^{-1}(U_i^{(j)})$ and $\Lambda=\alpha^{-1}(U)$ are commensurable, the group $\alpha^{-1}(U_i^{(j)})$ is of type $\FP^\bbQ$ (thus, $\FP^\bbC$). This follows from the combination of~\cite{bieri}*{Theorem~5.11 on p.~78} and~\cite{brown}*{Proposition 5.1 on p.~197 and Proposition 6.1 on p.~199}. 
Tensoring this resolution with $\bbC[\Gamma]$ we obtain a finite projective $\bbC[\Gamma]$-resolution of $\bbC[\Gamma/\alpha^{-1}(U_i^{(j)})]$ which we denote by $Q(i,j)_\ast$. For every $j\in\{0,\dots,n\}$, the sum $\bigoplus_{i\in I_j} Q(i,j)_\ast$ is a finite projective $\bbC[\Gamma]$-resolution of 
\[\res_\Gamma^G\bigl(\bigoplus_{i\in I_j} \bbC\bigl[G/U_i^{(j)}\bigr]\bigr)\cong \bigoplus_{i\in I_j} \bbC\bigl[\Gamma/\alpha^{-1}(U_i^{(j)})\bigr].\]
Similarly as in the proof of Theorem~\ref{thm: finiteness for discrete group}, we find a projective resolution $Q_\ast$ of the trivial $\bbC[\Gamma]$-module~$\bbC$ such that 
\[Q_n\cong\bigoplus_{k+j=n} \bigoplus_{i\in I_j}Q(i, j)_k.\]  

Using the compatibility of the von Neumann dimension under induction, we conclude that 
\begin{align*}
    \chi(\Gamma)&=\sum_{n\ge 0}(-1)^n\dim_{L(\Gamma)}\bigl(L(\Gamma)\otimes_{\bbC[\Gamma]}Q_n\bigr)\\
        &=\sum_{j\ge 0}(-1)^j \sum_{k\ge 0}(-1)^k \sum_{i\in I_j} \dim_{L(\Gamma)}\bigl(L(\Gamma)\otimes_{\bbC[\Gamma]}Q(i,j)_k\bigr)\\
        &=\sum_{j\ge 0}(-1)^j \sum_{i\in I_j}\sum_{k\ge 0}(-1)^k  \dim_{L(\alpha^{-1}(U_i^{(j)}))}\bigl(L(\alpha^{-1}(U_i^{(j)}))\otimes_{\bbC[\alpha^{-1}(U_i^{(j)})]}\tilde Q(i,j)_k\bigr)\\
        &=\sum_{j\ge 0}(-1)^j\sum_{i\in I_j}\chi^{(2)}\bigl(\alpha^{-1}(U_i^{(j)})\bigr)\\
        &=\sum_{j\ge 0}(-1)^j\sum_{i\in I_j}\mu\bigl(U_i^{(j)}\bigr)^{-1}\chi^{(2)}(\Lambda)\qquad\text{(use~\eqref{eq: multiplicativity for finite index} and $\mu(U)=1$ and $\Lambda=\alpha^{-1}(U)$)}\\
        &=\chi^{(2)}(G,\mu)\cdot\chi^{(2)}(\Lambda).\qedhere
\end{align*}

\end{proof}

\begin{example}
The group $\Gamma=\SL_n(\bbZ[1/p])$ is a lattice in $G=\SL_n(\bbR)\times \SL_n(\bbQ_p)$. 
Let $\Lambda=\SL_n(\bbZ)<\Gamma$. 
We compare Theorem~\ref{thm: euler characteristic} for $\Gamma\schlichting\Lambda$ to computations we obtain from the theory of $\ell^2$-Betti numbers of locally compact groups via~\eqref{eq: alternating l2 betti}.  
Let $\mu$ and $\nu$ be Haar measures of the left and right factor of $G$, respectively. 
Then 
\[
\chi^{(2)}\bigl( \Gamma\bigr) = \operatorname{covol}(\Gamma, \mu\times\nu)\cdot \chi^{(2)}\bigl( \SL_n(\bbR),\mu)\cdot\chi^{(2)}\bigl(\SL_n(\bbQ_p),\nu\bigr).
\]
Similarly, since $\SL_n(\bbZ)$ is a lattice of $\SL_n(\bbR)$ we obtain that 
\[ \chi^{(2)}\bigl( \SL_n(\bbZ)\bigr) = \operatorname{covol}\bigl(\SL_n(\bbZ), \mu\bigr)\cdot \chi^{(2)}\bigl( \SL_n(\bbR),\mu\bigr). 
\]
We normalize $\nu$ so that $\nu(\SL_n(\bbZ_p))=1$. The push-forward measure $\xi$ on $\PSL_n(\bbQ_p)$ under the projection $\SL_n(\bbQ_p)\to \PSL_n(\bbQ_p)$ satisfies $\xi(\PSL_n(\bbZ_p))=1$. By~\cite{petersen} and~\eqref{eq: alternating l2 betti} we have 
\[ \chi^{(2)}\bigl( \SL_n(\bbQ_p),\nu\bigr)=\chi^{(2)}\bigl( \PSL_n(\bbQ_p),\xi\bigr).\]
Therefore, 
\begin{equation}\label{eq: ratio of covolumes} \chi^{(2)}(\Gamma)=\frac{\covol(\Gamma,\mu\times\nu)}{\covol(\SL_n(\bbZ),\mu)}\cdot \chi^{(2)}\bigl(\SL_n(\bbZ)\bigr)\cdot \chi^{(2)}\bigl(\PSL_n(\bbQ_p),\xi\bigr).
\end{equation}

There is an isomorphism $\SL_n(\bbZ[1/p])\schlichting \SL_n(\bbZ)\cong \PSL_n(\bbQ_p)$ under which the closure of $\SL_n(\bbZ)$ is mapped onto $\PSL_n(\bbZ_p)$. See~\cite{shalom+willis}*{Example~3.10}. By Theorem~\ref{thm: euler characteristic}, 
\[
\chi^{(2)}(\Gamma)= \chi^{(2)}\bigl(\SL_n(\bbZ)\bigr)\cdot \chi^{(2)}\bigl(\PSL_n(\bbQ_p),\xi\bigr).
\]
As a consequence, the ratio of covolumes in~\eqref{eq: ratio of covolumes} is~$1$. 
\end{example}


\begin{bibdiv}
    \begin{biblist}
\bib{abels+brown}{article}{
   author={Abels, Herbert},
   author={Brown, Kenneth S.},
   title={Finiteness properties of solvable $S$-arithmetic groups: an
   example},
   booktitle={Proceedings of the Northwestern conference on cohomology of
   groups (Evanston, Ill., 1985)},
   journal={J. Pure Appl. Algebra},
   volume={44},
   date={1987},
   number={1-3},
   pages={77--83},
}
 
\bib{bieri}{book}{
   author={Bieri, Robert},
   title={Homological dimension of discrete groups},
   series={Queen Mary College Mathematics Notes},
   edition={2},
   publisher={Queen Mary College, Department of Pure Mathematics, London},
   date={1981},
   pages={iv+198},
}
\bib{borel+yang}{article}{
   author={Borel, Armand},
   author={Yang, Jun},
   title={The rank conjecture for number fields},
   journal={Math. Res. Lett.},
   volume={1},
   date={1994},
   number={6},
   pages={689--699},
}
\bib{brown-finiteness}{article}{
   author={Brown, Kenneth S.},
   title={Finiteness properties of groups},
   booktitle={Proceedings of the Northwestern conference on cohomology of
   groups (Evanston, Ill., 1985)},
   journal={J. Pure Appl. Algebra},
   volume={44},
   date={1987},
   number={1-3},
   pages={45--75},
}
\bib{brown}{book}{
    author={Brown, Kenneth S.},
    title={Cohomology of groups},
    series={Graduate Texts in Mathematics},
    volume={87},
    note={Corrected reprint of the 1982 original},
    publisher={Springer-Verlag, New York},
    date={1994},
    pages={x+306},
}
\bib{brown2}{article}{
   author={Brown, Kenneth S.},
   title={Complete Euler characteristics and fixed-point theory},
   journal={J. Pure Appl. Algebra},
   volume={24},
   date={1982},
   number={2},
   pages={103--121},
}
\bib{brown-acyclic}{article}{
   author={Brown, Kenneth S.},
   title={The geometry of finitely presented infinite simple groups},
   conference={
      title={Algorithms and classification in combinatorial group theory},
      address={Berkeley, CA},
      date={1989},
   },
   book={
      series={Math. Sci. Res. Inst. Publ.},
      volume={23},
      publisher={Springer, New York},
   },
   date={1992},
   pages={121--136},
}
\bib{brown+geoghegan}{article}{
   author={Brown, Kenneth S.},
   author={Geoghegan, Ross},
   title={An infinite-dimensional torsion-free ${\rm FP}\sb{\infty }$ group},
   journal={Invent. Math.},
   volume={77},
   date={1984},
   number={2},
   pages={367--381},
}

\bib{boudec}{article}{
   author={Le Boudec, Adrien},
   title={Compact presentability of tree almost automorphism groups},
   language={English, with English and French summaries},
   journal={Ann. Inst. Fourier (Grenoble)},
   volume={67},
   date={2017},
   number={1},
   pages={329--365},
}
\bib{ilaria+ged}{article}{
    author={Castellano, I.},
    author={Corob Cook, G.},
    title={Finiteness properties of totally disconnected locally compact groups},
    journal={J. Algebra},
    volume={543},
    date={2020},
    pages={54--97},
}
\bib{euler}{article}{
  author = {Castellano, Ilaria},
  author = {Chinello, Gianmarco},
  author = {Weigel, Thomas},
  title = {The Hattori-Stallings rank, the Euler-Poincaré characteristic and zeta functions of totally disconnected locally compact groups},
  eprint = {https://arxiv.org/pdf/2405.08105},
  year = {2024},
}

\bib{castellano+weigel}{article}{
    author={Castellano, I.},
    author={Weigel, Th.},
    title={Rational discrete cohomology for totally disconnected locally compact groups},
    journal={J. Algebra},
    volume={453},
    date={2016},
    pages={101--159},
}

\bib{geoghegan}{book}{
    author={Geoghegan, Ross},
    title={Topological Methods in Group Theory},
    series={Graduate Texts in Mathematics},
    publisher={Springer New York, NY},
    date={2008},
    pages={XVI+473},
}
\bib{guichardet}{book}{
   author={Guichardet, A.},
   title={Cohomologie des groupes topologiques et des alg\`ebres de Lie},
   language={French},
   series={Textes Math\'{e}matiques [Mathematical Texts]},
   volume={2},
   publisher={CEDIC, Paris},
   date={1980},
   pages={xvi+394},
}
\bib{tomdieck}{book}{
    author={tom Dieck, Tammo},
    title={Transformation groups},
    series={De Gruyter Studies in Mathematics},
    volume={8},
    publisher={Walter de Gruyter \& Co., Berlin},
    date={1987},
    pages={x+312},
}
\bib{kyed+petersen+vaes}{article}{
   author={Kyed, David},
   author={Petersen, Henrik Densing},
   author={Vaes, Stefaan},
   title={$L^2$-Betti numbers of locally compact groups and their cross
   section equivalence relations},
   journal={Trans. Amer. Math. Soc.},
   volume={367},
   date={2015},
   number={7},
   pages={4917--4956},
}
\bib{lueck}{book}{
   author={L\"{u}ck, Wolfgang},
   title={$L^2$-invariants: theory and applications to geometry and
   $K$-theory},
   series={Ergebnisse der Mathematik und ihrer Grenzgebiete. 3. Folge. A
   Series of Modern Surveys in Mathematics [Results in Mathematics and
   Related Areas. 3rd Series. A Series of Modern Surveys in Mathematics]},
   volume={44},
   publisher={Springer-Verlag, Berlin},
   date={2002},
   pages={xvi+595},
}
\bib{petersen}{article}{
   author={Petersen, Henrik Densing},
   title={$L^2$-Betti numbers of locally compact groups},
   language={English, with English and French summaries},
   journal={C. R. Math. Acad. Sci. Paris},
   volume={351},
   date={2013},
   number={9-10},
   pages={339--342},
}
\bib{petersen+sauer+thom}{article}{
   author={Petersen, Henrik Densing},
   author={Sauer, Roman},
   author={Thom, Andreas},
   title={$L^2$-Betti numbers of totally disconnected groups and their
   approximation by Betti numbers of lattices},
   journal={J. Topol.},
   volume={11},
   date={2018},
   number={1},
   pages={257--282},
}
\bib{sauer+thumann}{article}{
   author={Sauer, Roman},
   author={Thumann, Werner},
   title={Topological models of finite type for tree almost automorphism
   groups},
   journal={Int. Math. Res. Not. IMRN},
   date={2017},
   number={23},
   pages={7292--7320},
}

\bib{schlichting}{article}{
   author={Schlichting, G},
   title={Operationen mit periodischen Stabilisatoren},
   journal={Archiv der Mathematik (Basel)},
   volume={34},
   date={1980},
   number={2},
   pages={97-99},
}

\bib{shalom+willis}{article}{
   author={Shalom, Yehuda},
   author={Willis, George A.},
   title={Commensurated Subgroups of Arithmetic Groups, Totally Disconnected Groups and Adelic Rigidity},
   journal={Geom. Funct. Anal.},
   volume={23},
   date={2013},
   number={5},
   pages={1631--1683},
}
\bib{szymik+wahl}{article}{
   author={Szymik, Markus},
   author={Wahl, Nathalie},
   title={The homology of the Higman-Thompson groups},
   journal={Invent. Math.},
   volume={216},
   date={2019},
   number={2},
   pages={445--518},
}
\bib{tzanev}{article}{
   author={Kroum Tzanev},
   title={Hecke C*-algebras and amenability},
   journal={J. Operator Theory},
   volume={50},
   date={2003},
   number={1},
   pages={169-178},
}
         
\bib{weibel}{book}{
   author={Weibel, Charles A.},
   title={An introduction to homological algebra},
   series={Cambridge Studies in Advanced Mathematics},
   volume={38},
   publisher={Cambridge University Press, Cambridge},
   date={1994},
   pages={xiv+450},
}

    \end{biblist}
\end{bibdiv}
\end{document}